\documentclass[letterpaper,11pt,reqno]{amsart}

\usepackage{amsmath}%
\usepackage{amsfonts}%
\usepackage{amssymb}%
\usepackage{diagmac2} 
\usepackage{float}

\usepackage{scalerel}
\usepackage{verbatim}

\newtheorem{theorem}{Theorem}[section] 

\newtheorem{prop}[theorem]{Proposition}
\newtheorem{cor}[theorem]{Corollary}

\theoremstyle{plain}
\newtheorem{example}{Example}

\theoremstyle{definition}
\newtheorem{definition}{Definition}[section] 

\theoremstyle{remark}
\newtheorem{remark}{Remark}[section]

\newcommand{\maps}{\longrightarrow}
\newcommand{\Vars}{\mathcal{P}}
\newcommand{\Frm}{\mathsf{For}}
\newcommand{\LogL}{\mathsf{L}}
\newcommand{\LogP}{\mathsf{P}}

\newcommand{\zero}{\mathbf{0}}
\newcommand{\one}{\mathbf{1}}

\newcommand{\ruleR}{\mathsf{r}}
\newcommand{\Rules}{\mathsf{R}}

\newcommand{\bydef}{:=}
\newcommand{\means}{\leftrightharpoons}

\newcommand{\Heyt}{\mathcal{H}}
\newcommand{\Br}{\mathcal{B}}

\newcommand{\CPC}{\mathsf{CPC}}
\newcommand{\IPC}{\mathsf{IPC}}
\newcommand{\SF}{\mathsf{S5}}

\newcommand{\Ch}{\mathcal{C}}
\newcommand{\KC}{\mathcal{KC}}
\newcommand{\Int}{\mathbf{Int}}

\newcommand{\Var}[1]{\mathcal{#1}}

\newcommand{\var}{\Var{V}}
\newcommand{\DS}{\mathsf{S}}

\newcommand{\Alg}[1]{\mathbf{#1}}
\newcommand{\alg}[1]{\mathbf{#1}}
\newcommand{\el}[1]{\mathsf{#1}}

\newcommand{\F}{\mathsf{F}}

\newcommand{\Ladm}{\ \raisebox{2pt}{\scaleobj{.7}{|\!\!\!\sim}}_\LogL \ }

\newcommand{\Padm}{\ \raisebox{2pt}{\scaleobj{.7}{|\!\!\!\sim}}_{\LogP} \ }

\newcommand{\NLadm}{\ \raisebox{2pt}{\scaleobj{.7}{\not{|\!\!\!\sim}}}_\LogL \ }
\newcommand{\NPadm}{\ \raisebox{2pt}{\scaleobj{.7}{\not{|\!\!\!\sim}}}_{\LogP} \ }

\newcommand{\Bemb}{\rightarrowtail^+}

\newcommand{\set}[2]{\{#1 \! \mid \! #2 \}}
\newcommand{\Ext}[1]{\mathsf{Ext}#1}

\newcommand{\conc}[2]{[#1]_{#2}}

\title{Admissibility in Positive Logics}

\author{Alex Citkin}
\address{Metropolitan Telecommunications}
\email{acitkin@gmail.com}

\begin{document}

\begin{abstract}
The paper studies admissibility of multiple-conclusion rules in the positive logics. Using modification of a method used by M.~Wajsberg in the proof of the separation theorem, it is shown that the problem of admissibility in positive logics is equivalent to the problem of admissibility in intermediate logics defined by positive additional axioms.

\smallskip
\noindent \textbf{Keywords.}
Inference rule, multiple-conclusion rule, admissible rule, positive logic, intermediate logic, Brouwerian algebra.

\end{abstract}

\maketitle


\section{Introduction}

The notion of admissible rule evolved from the notion of auxiliary rule: if a formula $B$ can be derived from a set of formulas $A_1,\dots,A_n$ in a given calculus (deductive system) $\DS$, one can shorten derivations by using a rule $A_1,\dots,A_n/B$. The application of such a rule does not extend the set of theorems, i.e. such a rule is admissible (permissible). In \cite[p.19]{Lorenzen_Book_1955} P.~Lorenzen called the rules not extending the class of the theorems "zul\"assing," and the latter term was translated as "admissible," the term we are using nowadays. In \cite{Lorenzen_Protologik_1956} Lorenzen also linked the admissibility of a rule to existence of an elimination procedure. 

Independently, P.S.~Novikov, in his lectures on mathematical logic, had introduced the notion of derived rule: a rule $\mathcal{A}_1,\dots,\mathcal{A}_n/\mathcal{B}$, where $\mathcal{A}_1,\dots,\mathcal{A}_n,\mathcal{B}$ are variable formulas of some type, is derived in a calculus $\DS$ if $\vdash_\DS \mathcal{B}$ holds every time when $\vdash_\DS \mathcal{A}_1,\dots,\vdash_\DS \mathcal{A}_n$ hold (see \cite[p. 30]{Novikov_Book}\footnote{This book was published in 1977, but it is based on the notes of a course that P.S.~Novikov taught in 1950th; A.V.~Kuznetsov was recalling that P.S.~Novikov had used the notion of derivable rule much earlier, in this lectures in 1940th.}). And he distinguished between two types of derived rules: a derived rule is strong, if $\vdash_\DS \mathcal{A}_1 \to (\mathcal{A}_2 \to \dots (\mathcal{A}_n \to \mathcal{B}) \dots)$ holds, otherwise a derived rule is weak.       

For classical propositional calculus ($\CPC$), the use of admissible rules is merely a matter of convenience, for every admissible for $\CPC$ rule $A_1,\dots,A_n/B$ is derivable, that is $A_1,\dots,A_n \vdash_\CPC B$ (see, for instance \cite{Belnap_et_Strengthening_1963}). It was observed by R.~Harrop in \cite{Harrop_Concerning_1960} that the rule $\neg p \to (q \lor r)/(\neg p \to q) \lor (\neg p \to r)$ is admissible for the intuitionistic propositional calculus ($\IPC$), but is not derivable in $\IPC$. Later, in mid 1960s, A.V.~Kuznetsov observed that the rule $(\neg\neg p \to p) \to (p \lor \neg p)/((\neg\neg p \to p) \to \neg p) \lor ((\neg\neg p \to p) \to \neg\neg p)$ is also admissible for $\IPC$, but not derivable. Another example of an admissible for $\IPC$ not derivable rule was found in 1971 by G.~Mints (see \cite{Mints1971}). Moreover, it was observed in \cite{Citkin1977} that there is an infinite set of independent rules admissible and not derivable in $\IPC$. And it was established by V.~Rybakov (see \cite{Rybakov_Criterion_Adm_1984,Rybakov_Bases_Adm_1985}) that there is no finite basis of admissible in $\IPC$ (and $\SF$) rules, i.e. not all admissible in $\IPC$ rules can be derived from any given finite set of admissible in $\IPC$ rules.

Naturally, the question about admissibility of rules in the extensions of $\IPC$ and their fragments arose. And in \cite{Prucnal_On_Structural_1972} T.~Prucnal had proven that $\to$-fragment of any superintuitionistic logic, is structurally complete, i.e. these fragments do not have admissible not-derivable rules. The Prucnal's proof method can be extended to any $\lor$-free fragment of any superintuitionistic.

At about the same time the author observed \cite{Citkin_Positive_1988} that admissible rules of the positive logics can be reduced to rules of some superintuitionistic logics. Namely, it was observed that a rule $\ruleR \bydef A/B$, where $\neg$ does not occur in formulas $A$ and $B$,  is admissible in a positive logic $\LogP$ if and only if $\ruleR$ is admissible in the least superintuitionistic logic $\LogL$ having $\LogP$ as its positive fragment. The admissibility in various positive and paraconsistent logics was studied in details in \cite{Odintsov_Rybakov_Unification_2013}. 

In the review \cite{Kracht_Review_1999} of book \cite{Rybakov_Book}, M.~Kracht suggested to study admissibility of multiple-conclusion rules: a rule $A_1,\dots,A_n/B_1,\dots,B_n$ is admissible in a logic $\LogL$ if every substitution that makes all the premises valid in $\LogL$, makes at least one conclusion valid in $\LogL$ (see also \cite{Kracht_Modal_2007}). A natural example of multiple-conclusion rule (called m-rule for short) admissible in $\IPC$ is the rule representing the disjunction property: $A \lor B/A,B$. That is, if a formula $A \lor B$ is valid in $\IPC$, then at least one of the formulas $A,B$ is valid in $\IPC$. The bases of admissibile m-rules for a variety of superintuitionistic and normal modal logics were constructed in \cite{Jerabek_Independent_2008,Jerabek_Canonical_2009,Goudsmit_Iemhoff_Unification_2014,Goudsmit_PhD}. In \cite{Cintula_Metcalfe_Admissible_2010} the admissibility in $\to,\neg$-fragment of $\IPC$ has been studied. 

In the present paper we consider admissibility of m-rules in positive logics, which are precisely (see Corollary \ref{corposf}) the positive fragments of superintuitionistic logics, and we show (Theorem \ref{thfollow}) that admissibility of any positive m-rule in a given positive logic $\LogP$ can be reduced to admissibility of this rule for  superintuitionistic logic $\Int + \LogP$.

\section{Background}

\subsection{Logics}

The (propositional) formulas are built in a regular way from a countable set $\Vars$ of (propositional) variables and connectives $\land, \lor,\to, \zero$. The set of all formulas is denoted by $\Frm$. By $\Frm^+$ we denote the subset of all formulas not containing $\zero$ and we call these formulas "\emph{positive}". If $A$ is a formula, $\pi(A)$ is a set of all variables occurring in $A$. 

We use $\Sigma$ to denote the set of all substitutions on $\Frm$, that is, $\Sigma$ is a set of all mappings $\sigma: \Vars \longrightarrow \Frm$, while $\Sigma^+$ denotes a set of all \emph{positive substitutions}, that is,  the set of all mappings $\sigma^+: \Vars \longrightarrow \Frm^+$.

\textit{Superintuitionistic logics} (si-logics for short) are understood as sets of formulas containing all theorems of intuitionistic propositional logic (denoted by $\Int$) and closed under rules modus ponens and substitution. The set of all si-logics is denoted by $\Ext{\Int}$ and forms a complete lattice with the least element $\Int$ and the unit $\Frm$. 

If $\LogL$ is an si-logic, by $\LogL^+$ we denote a \textit{positive fragment} of $\LogL$, that is, $\LogL^+ \bydef \LogL \cap \Frm^+$. 

\textit{Positive (superintuitionistic) logics} are understood as sets of positive formulas extending $\Int^+$ and closed under rules modus ponens and positive substitutions. The set $\Ext{\Int^+}$ of all positive logics also forms a complete lattice. Clearly, positive fragment of each si-logic is a positive logic, and 
\begin{equation}
\varphi: \LogL \maps \LogL^+ \label{posmap}
\end{equation}
is a homomorphism of $\Ext{\Int}$ to $\Ext{\Int^+}$. 

For any set of formulas (of positive formulas) $\Gamma$, by $\Int + \Gamma$ (or by $\Int^+ + \Gamma$) we denote the least si-logic containing $\Int \cup \Gamma$ (the least positive logic containing $\Int^+ \cup \Gamma$).

For us, the following relations between si- and positive logics are important:
\begin{itemize}
\item[(a)] For any set $\Gamma$ of positive formulas
\begin{equation}
(\Int + \Gamma)^+ = \Int^+ + \Gamma \label{posfrag}
\end{equation}
(see Corollary \ref{corsep});
\item[(b)] $\varphi$ from \eqref{posmap} maps $\Ext{\Int}$ onto $\Ext{\Int^+}$, that is, every positive logic is a positive fragment of some si-logic (see {\cite{Verhozina_Intermediate_1978}[Proposition 1]}]);

\item[(c)] For every positive logic $\LogP$ there is the least si-logic that has $\LogP$ as its positive fragment, namely, $\Int + \LogP$ is such a logic (comp. \cite{Verhozina_Intermediate_1978}[Corollary p.16]). 
\end{itemize}
Let us note that (a) entails (b) and (c) if we take $\Gamma = \LogP$ and recall that $\Int^+ \subseteq \LogP$:
\[
(\Int + \LogP)^+ = \Int^+ + \LogP = \LogP; 
\] 
and, obviously, $\Int + \LogP$ is the least si-logic containing $\LogP$.

\subsection{Admissible Rules}

An ordered pair $\Gamma/\Delta$ of finite sets of formulas $\Gamma,\Delta \subseteq \Frm$ is called a \textit{multiple-conclusion rule} (m-rule for short). If sets $\Gamma,\Delta$ consist only of positive formulas, the m-rule $\Gamma/\Delta$ is called a \textit{positive m-rule}. The formulas from $\Gamma$ are \textit{premises} of the rule, while the formulas from $\Delta$ are \textit{conclusions} of the rule.

Formula $A$ is said to be \textit{unifiable in a logic} $\LogL$ ($\LogL$-unifiable for short), if there is a substitution $\sigma \in \Sigma$, called $\LogL$-\textit{unifier of} $A$, application of which makes $A$ a theorem, that is, such  that $\sigma(A) \in \LogL$. Respectively, a positive formula $A$ is unifiable in a positive logic $\LogP$ if there is a $\LogP$-unifier, that is, there is a positive substitution $\sigma^+$ such that $\sigma^+(A) \in \LogP$.

A finite set of formulas (of positive formula) $\Gamma$ is $\LogL$-unifiable ($\LogP$-unifiable) if there is a substitution (a positive substitution) that unifies all formulas from $\Gamma$. 

\begin{definition} Let $\LogL$ be a logic. An m-rule (a positive m-rule) $\Gamma/\Delta$ is \textit{admissible} for $\LogL$ (respectively, for $\LogP$), if every substitution that unifies $\Gamma$ unifies at least one formula from $\Delta$. Respectively, if $\LogP$ is a positive logic, $\Gamma/\Delta$ is admissible for $\LogP$ if every $\LogP$-unifier of $\Gamma$ unifies at least one formula from $\Delta$. If $\Gamma = \varnothing$ every substitution unifies $\Gamma$. And the rule $\varnothing/\varnothing$ is admissible only in the \textit{inconsistent logic} $\Frm$. 
\end{definition}

If an m-rule $\ruleR$ is admissible for a logic $\LogL$, we denote this by $\Ladm \ruleR$. If a positive m-rule $\ruleR$ is admissible for a positive logic $\LogP$ we denote this by $\Padm \ruleR$.

Let us note that $\Ladm \varnothing/\Delta$ means that $\Delta \cap \LogL \neq\varnothing $, while $\Ladm \Gamma/\varnothing$ means that $\Gamma \neq \varnothing$ and $\Gamma$ is not $\LogL$-unifiable set of formulas. The rule $\varnothing/\varnothing$ is not admissible in any logic.

\begin{example} Rule $(p \to q) \to (p \lor r)/((p \to q) \to p) \lor ((p \to q) \to r)$ is admissible for $\Int$  (see \cite{Mints1971}), and, hence, is admissible for $\Int^+$.
\end{example}

\subsection{Algebras}

In this section we consider the algebraic models for superintuitionistic and positive logics.

The algebraic models for positive logics are Brouwerian algebras\footnote{There are different names used for Brouwerian algebras: lattices with relative pseudo-complementation  (see e.g. \cite{Birkhoff}), generalized Brouwerian algebras (see e.g. \cite{Monteiro_Axioms_1955}), implicative lattices (see e.g. \cite{Jankov_Calculus_1968,Odintsov_Constructive_2008}).}. A \textit{Brouwerian algebra} (see e.g. \cite{Kohler_Varieties_1975}) is an algebra $(\alg{A},\land,\lor,\to, \one)$ in which $(\alg{A},\land,\lor)$  is a distributive lattice with unit $\one$ and relative pseudo-complementation $\to$. Class $\Br$ of all Brouwerian algebras forms a variety. 

The algebraic models for si-logics are Heyting algebras. A \textit{Heyting algebra} is an algebra $(\alg{A},\land,\lor,\to, \one, \zero)$, where $(\alg{A},\land,\lor,\to, \one)$ is a Brouwerian algebra and $\zero$ is the least element, that is the identity $\zero \to x \approx \one$ holds. By $\Heyt$ we denote the variety of all Heyting algebras.

Map $\nu: \Vars \maps \Alg{A}$ is called a \textit{valuation} in a given algebra $\Alg{A}$. In a natural way, $\nu$ can be extended to map $\nu: \Frm \maps \Alg{A}$. If $A \in \Frm$ and $\nu(A) = \one_\Alg{A}$ we say that $\nu$ \textit{validates} $A$ in $\Alg{A}$, otherwise we say that $\nu$ \textit{refutes} $A$ in $\Alg{A}$.

As usual, if $\Alg{A}$ is an algebra and $A$ is a formula, $\Alg{A} \models A$ means that formula $A$ is \textit{valid in algebra} $\Alg{A}$, that is, every valuation in $\Alg{A}$ validates $A$.

If $\ruleR \bydef \Gamma/\Delta$ is an m-rule, we say that $\ruleR$ is \textit{valid in an algebra}$\Alg{A}$ (in symbols, $\Alg{A} \models \ruleR$), if every valuation that validates in $\Alg{A}$ all premises, validates at the same time at least one of conclusions. If $\Gamma = \varnothing$, every valuation validates $\Gamma$. On the other hand, if $\Delta = \varnothing$, neither valuation validates $\Delta$. Thus, the rule $\varnothing/\varnothing$ is not valid in any algebra, while the rule $\zero/\varnothing$ is valid in any non-degenerate (that is, having more than one element) algebra.

If $\Rules$ is a set of m-rules and $\Alg{A}$ is an algebra, $\Alg{A} \models \Rules$ means that every m-rule from $\Rules$ is valid in $\Alg{A}$. On the other hand, if $\Var{S}$ is a set of algebras and $\ruleR$ is an m-rule, $\Var{S} \models \ruleR$ means that $\ruleR$ is valid in every algebra from $\Var{S}$. 

Each si-logic $\LogL$ (or each positive logic $\LogP$) has a \textit{corresponding variety} $\var_\LogL$ (or $\var_\LogP$) of Heyting (or Brouwerian) algebras:
\[
\var_\LogL \bydef \set{\Alg{A} \in \Heyt}{\Alg{A} \models A, \text{ for every } A \in \LogL}
\]
or
\[
\var_\LogP \bydef \set{\Alg{A} \in \Br}{\Alg{A} \models A, \text{ for every } A \in \LogP}.
\]

For a Heyting algebra $\Alg{A}$ by $\Alg{A}^+$ we denote a \textit{Brouwerian reduct} of $\Alg{A}$, that is, $\{\land,\lor,\to,\one\}$-reduct of $\Alg{A}$. A Heyting algebra $\Alg{A}$ is \textit{B-embedded} in a Heyting algebra $\Alg{B}$ (in written $\Alg{A} \Bemb \Alg{B}$) if $\Alg{A}^+$ is embedded into $\Alg{B}^+$ as a Brouwerian algebra. If $\var$ is a variety of Heyting algebras and $\Alg{A}$ is B-embedded in an algebra from $\var$, we denote this by $\Alg{A} \in^+ \var$.

\subsection{Congruences} Let us recall (see e.g. \cite{Rasiowa_Algebraic_1974}) that a non-void set  $\F$ of elements of a given Brouwerian algebra $\Alg{A}$ (or a given Heyting algebra) is a \textit{filter} of $\Alg{A}$ if $\one \in \F$ and $\alg{a},\alg{a} \to \alg{b} \in \F$ yields $\alg{b} \in \F$ for all elements $\alg{a},\alg{b} \in \Alg{A}$. It is not hard to demonstrate that a meet of an arbitrary set of filters of a given algebra is a filter. If $\Alg{A}$ is a Brouwerian algebra (a Heyting algebra) and $\el{A} \subseteq \Alg{A}$ there is the least filter of $\Alg{A}$ (denoted by $[\el{A})$) containing $\el{A}$, and $\el{A}$ is said to be a filter \textit{generated} by elements $\el{A}$. There is a close connection between filters and congruences: every filter $\F$ of a given Brouwerian algebra (or Heyting algebra) defines a congruence that we denote $\theta(\F)$
\[
(\alg{a},\alg{b}) \in \theta_\F \text{ if and only if } \alg{a} \to \alg{b}, \alg{b} \to \alg{a} \in \F.
\] 
On the other hand, every congruence $\theta$ of a given Brouwerian algebra (Heyting algebra) $\Alg{A}$ defines a filer $\F_\theta$:
\[
\F_\theta \bydef \set{\alg{a} \in \Alg{A}}{(\alg{a},\one) \in \theta}.
\]
And the following holds:
\[
\theta(\F_\theta) = \theta \text{ and } \F_{\theta(\F)} = \F.
\]

Since every filter $\F$ of an algebra $\Alg{A}$ defines a congruence $\theta(\F)$, we use $\Alg{A}/\F$ to denote quotient algebra $\Alg{A}/\theta(\F)$.

Let us also observe that for each homomorphism $\varphi: \Alg{A} \maps \Alg{B}$ the set $\set{\alg{a} \in \Alg{A}}{\varphi(\alg{a}) = \one_\Alg{B}}$ forms a filter of $\Alg{A}$ that we denote $\F_\varphi$ and
\begin{equation}
\Alg{A}/\F_\varphi \cong \Alg{B}. \label{filterhom}
\end{equation}

Given an algebra $\Alg{A}$ and a congruence $\theta$ of $\Alg{A}$ by $\conc{\alg{a}}{\theta}$ we denote a congruence class containing element $\alg{a}$, that is, 
\[
\conc{\alg{a}}{\theta} \bydef \set{\alg{b} \in \Alg{A}}{(\alg{a},\alg{b}) \in \theta} 
\]
If $\varphi$ is a homomorphism of $\Alg{A}$, $\conc{\alg{a}}{\varphi}$ denotes the congruence class $\conc{\alg{a}}{\theta}$, where $\theta$ is a kernel congruence of $\varphi$ (that is, such a congruence that $\Alg{A}/\theta \cong \varphi(\Alg{A})$).

Since filters, and hence congruences, of Heyting algebra are defined by operation $\to$ which Heyting algebras share with their Brouwerian subreducts, the congruences of Heyting algebras and their Brouwerian subreducts are very well coordinated. Specifically, the following holds.

\begin{prop} \label{prhom} Let $\Alg{A}$ be a Heyting algebra and $\tilde{\Alg{A}}$ be a Brouwerian subalgebra of $\Alg{A}^+$ (Brouwerian subreduct of $\Alg{A}$). Then the following holds:
\begin{itemize}
\item[(a)] if $\varphi: \Alg{A} \maps \Alg{B}$ is a homomorphism, then $\varphi(\tilde{\Alg{A}})$ is a Brouwerian subreduct of $\Alg{B}$; 
\item[(b)] if $\F$ is a filter of $\tilde{\Alg{A}}$, then
\begin{equation}
\F = \tilde{\Alg{A}} \cap [\F), \label{prhom_eq1}
\end{equation}
i.e. the filter of $\Alg{A}$ generated by elements of $\F$ does not contain elements of $\tilde{\Alg{A}}$ not belonging to $\F$.
\end{itemize}
\end{prop} 
The proof is an easy exercise and it is left to the reader. 

We will use the following consequence of the above Proposition.

\begin{theorem} \label{thhomsubr} Let $\Alg{A}$ be a Heyting algebra and $\Alg{B} \subseteq \Alg{A}^+$ be a Brouwerian subreduct of $\Alg{A}$. Then every Brouwerian homomorphic image of $\Alg{B}$ is (isomorphic to) a Brouwerian subreduct of a suitable homomorphic image of $\Alg{A}$.
\end{theorem}
\begin{proof}
Suppose $\varphi: \Alg{B} \maps \Alg{C}$ is a Brouwerian epimorphism. Then $\F_\varphi$ is a filter of $\Alg{B}$. Let us consider filter $[\F_\varphi)$ of $\Alg{A}$ and a homomorphism
\[
\psi: \Alg{A} \maps \Alg{A}/[\F_\varphi).
\]

From Proposition \ref{prhom}(a) we know that $\psi(\Alg{B})$ is a homomorphic image of $\Alg{B}$ and, hence, $\psi(\Alg{B})$ is a Brouwerian subreduct of $\Alg{A}/[\F_\varphi)$. Let us observe that, due to $\F \subseteq [\F)$, we have $\conc{\alg{b}}{\varphi} \subseteq \conc{\alg{b}}{\psi}$. Hence, all we need to prove is that $\conc{\alg{a}}{\varphi} \neq \conc{\alg{b}}{\varphi}$ yields $\conc{\alg{a}}{\psi} \neq \conc{\alg{b}}{\psi}$ for any $\alg{a},\alg{b} \in \Alg{B}$.

Let us prove the contrapositive statement: if $\alg{a},\alg{b} \in \Alg{B}$ and $\conc{\alg{a}}{\psi} = \conc{\alg{b}}{\psi}$, then $\conc{\alg{a}}{\varphi} = \conc{\alg{b}}{\varphi}$. Suppose that $\alg{a},\alg{b} \in \Alg{B}$ and $\conc{\alg{a}}{\psi} = \conc{\alg{b}}{\psi}$. Then
\begin{equation}
\alg{a} \to \alg{b} \in [\F) \text{ and } \alg{b} \to \alg{a} \in [\F). \label{thhom_eq1}
\end{equation}
Due to $\alg{a},\alg{b} \in \Alg{B}$ and $\Alg{B}$ is a Brouwerian algebra and, hence, it is closed under $\to$, we have $\alg{a}\to \alg{b},\alg{b} \to \alg{a} \in \Alg{B}$. Thus, by \eqref{prhom_eq1},
\begin{equation}
\alg{a} \to \alg{b} \in \F \text{ and } \alg{b} \to \alg{a} \in \F. \label{thhom_eq2}
\end{equation}
Hence $\conc{\alg{a}}{\varphi} = \conc{\alg{b}}{\varphi}$.
\end{proof}

\begin{cor} Let $\Alg{A}$ be a Heyting algebra, $\Alg{B}$ be a homomorphic image of $\Alg{A}$ and $\var$ be a variety of Heyting algebras. Then $\Alg{B} \in^+ \var$ as long as $\Alg{A} \in^+ \var$. In other words, class of Heyting algebras B-embedded in a given variety is closed under the formation of homomorphic images.
\end{cor}
\begin{proof} Suppose that $\Alg{A} \in^+ \var$ and $\varphi: \Alg{A} \maps \Alg{B}$ is an epimorphism. Then there is a Heyting algebra $\Alg{C} \in \var$ having a Brouwerian subreduct $\tilde{\Alg{C}}$ isomorphic to $\Alg{A}^+$. Let $\psi: \tilde{\Alg{C}} \maps \Alg{A}^+$ be an isomorphism. By Proposition \ref{prhom}(a), $\varphi(\Alg{A}^+)$ is a Brouwerian subreduct of $\Alg{B}$ and, due to $\psi$ is an onto mapping, $\varphi(\Alg{A}^+) = \Alg{B}^+$. Consider the homomorphism $\varphi': \alg{c} \mapsto \varphi(\psi(\alg{c})), \alg{c} \in \tilde{\Alg{C}}$. 
\begin{figure}[!ht]
\[
\ctdiagram{
\ctv 0,60:{\Alg{A}^+}
\ctv 60,60:{\Alg{B}^+}
\ctv 0,0:{\tilde{\Alg{C}}}
\ctet 0,60,60,60:{\varphi}
\ctel 0,0,0,60:{}
\ctel 0,60,0,0:{\psi}
\cter 0,0,60,60:{\varphi'}
}
\]
\end{figure}
By Theorem \ref{thhomsubr}, $\varphi'(\tilde{\Alg{C}})$ is a Brouwerian subreduct of some homomorphic image $\Alg{C}'$ of algebra $\Alg{C}$. Due to $\var$ being a variety, it is closed under the formation of homomorphic images, $\Alg{C}' \in \var$. Thus, $\Alg{B}^+$ is isomorphic to a Brouwerian subreduct of $\Alg{C}'$ and $\Alg{C}' \in \var$, that is, $\Alg{B} \in^+ \var$.  
\end{proof}

If $\var$ is a variety of Heyting algebras, by $\var^+$ we denote the set of all Brouwerian subreducts of all algebras from $\var$, that is,
\begin{equation}
\var^+ \bydef \set{\Alg{B} \in \Br}{\Alg{B} \Bemb \Alg{A} \text{ for some } \Alg{A} \in \var}. \label{subred}
\end{equation}
Note that, by Maltsev's embedding theorem (see e.g. \cite{GorbunovBookE}[Proposition 1.5.1]), $\var^+$ always forms a quasivariety.

\begin{cor} \label{thsubv} Let $\var$ be a variety of Heyting algebras . Then $\var^+$ forms a variety.
\end{cor}
\begin{proof} Suppose that $\var$ is a variety of Heyting algebras. Then, $\var^+$ is a quasivariety and, hence, is closed under the formation of subalgebras and direct products. So, we only need to prove that $\var^+$ is closed under the formation of the homomorphic images.

Indeed, let $\Alg{B} \in \var^+$. Then, by the definition of $\var^+$, there is an algebra $\Alg{A} \in \var$ and $\Alg{B}$ is B-embedded in $\Alg{A}$. By Theorem \ref{thhomsubr}, every homomorphic image of $\Alg{B}$ is B-embedded in a homomorphic image of $\Alg{A}$ which belongs to $\var$, for $\var$ is a variety and as such $\var$ is closed under the formation of  homomorphic images. Hence, every homomorphic image of $\Alg{B}$ is a member of $\var^+$. 
\end{proof}

\section{B-Saturated Varieties}

In this section we introduce and study B-saturated varieties of Heyting algebras that play a central role in what follows. 

First, let us make a simple but important observation (comp. \cite{Horn_Separation_1962,Kohler_Varieties_1975,Verhozina_Intermediate_1978}) that plays a role in the sequel of this paper.  

\begin{prop} \label{leastel}  If $\Alg{A} \in \Br$ is a finitely generated Brouwerian algebra, then $\Alg{A}$ contains the least element and, therefore, $\Alg{A}$ forms a Heyting algebra.
\end{prop}
\begin{proof} Let elements $\alg{a}_1,\dots,\alg{a}_n \in \Alg{A}$ generate algebra $\Alg{A}$. By a straightforward induction on the length of a formula expressing an element via generators, it is not hard to prove that element $\alg{a} \bydef \alg{a}_1 \land \dots \land \alg{a}_n$ is the least element of $\Alg{A}$. Indeed,  $\alg{a} \leq \alg{a}_i$ for all $i=1,\dots,n$, and $\alg{a} \leq \alg{b}$ and $\alg{a} \leq \alg{c}$ yields $\alg{a} \leq \alg{b} \circ \alg{c}$ for all $\circ \in \{\land,\lor,\to\}$.
\end{proof}

Let $\Alg{A}$ be a Brouwerian algebra. If $\Alg{A}$ contains the least element, it forms a Heyting algebra and we denote it by $\Alg{A}^\zero$. If $\Alg{A}$ does not contain the least element, by $\Alg{A}^\zero$ we denote a Heyting algebra obtained from $\Alg{A}$ by adjoining a new element $\zero$ and defining the operations in the following way: $\alg{a} \land \zero = \zero \land \alg{a} = \zero; \alg{a} \lor \zero = \zero \lor \alg{a} = \alg{a}; \zero \to \alg{a} = \one; \alg{a} \to \zero = \zero$.  

Now we are in a position to prove \eqref{posfrag}.

\begin{cor}\label{corsep}  For any set $\Gamma$ of positive formulas
\[
(\Int + \Gamma)^+ = \Int^+ + \Gamma.
\]
\end{cor}
\begin{proof} Indeed,
\[
\Int + \Gamma \supseteq \Int^+ + \Gamma,
\]
hence,
\[
(\Int + \Gamma)^+ \supseteq (\Int^+ + \Gamma)^+ = \Int^+ + \Gamma.
\]

Conversely, assume for contradiction that there is a positive formula $A$ such that 
\begin{equation}
A \in (\Int + \Gamma)^+  \text{ and } A \notin \Int^+ + \Gamma. \label{eq_corsep0}
\end{equation}
Then, due to $A \notin \Int^+ + \Gamma$, there is a finitely generated Brouwerian algebra $\Alg{A}$ such that
\begin{equation}
\Alg{A} \models \Gamma \text{ and } \Alg{A} \not\models A. \label{eq_corsep1}
\end{equation}
 By Proposition \ref{leastel}, $\Alg{A}$ forms a Heyting algebra $\Alg{A}^\zero$. Due to $\Alg{A}^\zero$ is a Heyting algebra, we have $\Alg{A}^\zero \models \Int$; and, due to all formulas from $\Gamma$ are positive, from \eqref{eq_corsep1} we can conclude that $\Alg{A}^\zero \models \Gamma$. Thus, $\Alg{A}^\zero \models (\Int \cup \Gamma)$ and, hence, $\Alg{A}^\zero \models (\Int + \Gamma)^+$, and, by \eqref{eq_corsep0}, $A \in (\Int + \Gamma)^+$, hence 
\begin{equation}
\Alg{A}^\zero \models A.  \label{eq_corsep2}
\end{equation}
On the other hand, since $A$ is a positive formula, any refuting valuation in $\Alg{A}$ is at the same time a refuting valuation in $\Alg{A}^\zero$, for algebras $\Alg{A}$ and $\Alg{A}^\zero$ have the same universes. Thus, from \eqref{eq_corsep1} we have $\Alg{A}^\zero \not\models A$ and we have arrived at contradiction with \eqref{eq_corsep2}. 
\end{proof}

\begin{cor} \label{corposf} Let $\LogP$ be a positive logic. Then
\begin{equation}
\LogP = (\Int + \LogP)^+.  \label{eqcorposf}
\end{equation} 
\end{cor}
\begin{proof} Indeed, by Corollary \ref{corsep}, taking  $\Gamma = \LogP$, we have
\[
\Int^+ + \LogP = (\Int + \LogP)^+,  
\]
and $\Int^+ + \LogP = \LogP$ due to $\Int^+ \subseteq \LogP$. 
\end{proof}

\begin{cor} \label{cordist} For any positive logics $\LogP_0$ and $\LogP_1$, 
\begin{equation}
\LogP_0 = \LogP_1 \text{ if and only if } \Int + \LogP_0 = \Int + \LogP_1. \label{eqcordist}
\end{equation} 
\end{cor}
\begin{proof} If $\LogP_0 \neq \LogP_1$, by Corollary \ref{corposf}, $(\Int + \LogP_0)^+ \neq (\LogP_1 + \Int)^+$ and, hence $(\Int + \LogP_0) \neq (\LogP_1 + \Int)$. And the converse statement is trivial.
\end{proof}

\subsubsection{Definition of B-Saturated Variety}

\begin{definition}
We say that a variety $\var$ of Heyting algebras is \emph{B-saturated} if for every Heyting algebra $\Alg{A}$
\[
\Alg{A} \in^+ \var \text{ entails } \Alg{A} \in \var.
\]
\end{definition}

The following proposition gives an intrinsic characterization of B-saturated varieties.

\begin{prop} \label{printr} Let $\var$ be a variety of Heyting algebras. Then the following is equivalent:
\begin{itemize}
\item[(a)] $\var$ is B-saturated; 
\item[(b)] for any algebra $\Alg{A} \in \var$ and any elements $\alg{a}_1,\dots,\alg{a}_n \in \Alg{A}$, if $\Alg{B}$ is a Brouwerian subalgebra of $\Alg{A}^+$ generated (as a Brouwerian subalgebra) by these elements, then $\Alg{B}^\zero \in \var$.
\end{itemize}
\end{prop}

Condition (b) can be rephrased in the following way: every finitely generated Brouwerian subreduct of any algebra from $\var$ belongs to $\var$ if regarded as a Heyting algebra. 

\begin{proof} (a) $\Rightarrow$ (b). Assume that $\var$ is a B-saturated variety of Heyting algebras, $\Alg{A} \in \var$ and $\alg{a}_i,i < n$ are elements of $\Alg{A}$. Let $\Alg{B}$ be a Brouwerian subalgebra of $\Alg{A}^+$ generated by elements $\alg{a}_i,i < n$ as Brouwerian algebra.  We need to show that $\Alg{B}^\zero \in \var$.

Indeed, let us note that $\Alg{B}^\zero$ is B-embedded in $\Alg{A}$ and, due to $\var$ is B-saturated, we have $\Alg{B}^\zero \in \var$.

(b) $\Rightarrow$ (a). Suppose that for any algebra $\Alg{A} \in \var$ and any elements $\alg{a}_1,\dots,\alg{a}_n \in \Alg{A}$, if $\Alg{B}$ is a Brouwerian subalgebra of $\Alg{A}^+$ generated (as a Brouwerian subalgebra) by these elements, then $\Alg{B}^\zero \in \var$.

Now, assume for contradiction that for some Heyting algebra $\Alg{A}$,
\begin{equation}
 \Alg{A} \in^+ \var \text{ and } \Alg{A} \notin \var. \label{eq_printr0}
\end{equation}
Then there is a formula $A(p_0,\dots,p_{n-1})$ such that 
\begin{equation}
\var \models A \text{ and } \Alg{A} \not\models A. \label{eq_printr1}
\end{equation} 
Suppose that $\alg{a}_i,i < n$ are refuting elements, that is, $A(\alg{a}_0,\dots, \alg{a}_{n-1}) \neq \one$. Let us consider Brouwerian subalgebra $\Alg{B}$ of $\Alg{A}^+$ generated as Brouwerian algebra by elements $\zero$ and $\alg{a}_i, i <n$. On the one hand, by assumption, $\Alg{B}^\zero \in \var$. On the other hand, due to $\zero \in \Alg{B}$, we have $\Alg{B} = (\Alg{B}^\zero)^+$, therefore $\Alg{B}^\zero \not\models A$. Thus, $\var \not\models A$ and this contradicts \eqref{eq_printr1}.
\end{proof}

We will see (Corollary \ref{corcont}) that there is continuum many distinct B-suturated subvarieties of $\Heyt$.

\begin{example} Clearly, the variety $\Heyt$ of all Heyting algebras is B-saturated, for, by Proposition \ref{leastel}, every finitely generated Brouwerian algebra forms a Heyting algebra and $\Heyt$ contains every Heyting algebra. 
\end{example}

Let $\Heyt_n$ be a variety of all Heyting algebras not containing chain (that is, lineary ordered) subalgebras having $n+2$ element. For instance, $\Heyt_1$ is a variety of all Boolean algebras.

\begin{example} Variety $\Heyt_n$ is B-saturated for all $n > 0$. Indeed, for any algebra $\Alg{A}$ any Brouwerian chain subalgebra of $\Alg{A}^+$ can be extended (by adjoining $\zero$, if necessary) to a Heyting chain subalgebra of $\Alg{A}$.
\end{example}

\begin{example} Variety $\Ch$ generated by all chain algebras is B-saturated. The proof is pretty straightforward and it is based on the fact that only the chain algebras are subdirectly irreducible in $\Ch$, and we leave the proof to the reader.
\end{example}

On the other hand, variety $\KC$ defined by additional axiom $\neg x \lor \neg\neg x \approx \one$ is not B-saturated: algebra depicted at Fig. \ref{figKC} belongs to $\KC$, while positive subalgebra generated by elements $\alg{a},\alg{b}$ (elements of which are marked by $\circ$) does not belong to $\KC$.

\begin{figure}[ht]
\[
\ctdiagram{
\ctnohead
\ctinnermid
\ctel 0,40,20,20:{}
\ctel 0,40,20,60:{}
\ctel 40,40,20,60:{}
\ctel 20,60,20,80:{}
\ctel 20,20,40,40:{}
\ctel 20,20,20,0:{}
\ctv 0,40:{\circ}
\ctv 20,21:{\circ}
\ctv 20,60:{\circ}
\ctv 40,40:{\circ}
\ctv 20,80:{\circ}
\ctv 20,-1:{\bullet}
\ctv 48,40:{\alg{b}}
\ctv -8,40:{\alg{a}}
}
\]
\caption{} \label{figKC}
\end{figure}

It is not hard to see that any subvariety of $\KC$ containing the above algebra is not B-saturated (note that subvarieties of $\KC$ not containing this algebra are precisely the ones generated by some chain algebras).

\section{Wajsberg Reduction}

In \cite{Wajsberg_Heyting_1977}, as a part of the proof of separation theorem for $\IPC$, Wajsberg had observed the close relations between $\Int$ and $\Int^+$. More precisely, he came with a way to link every formula $A$ with a positive formula $A^+$ in such a way that $A \in \Int$ if and only if $A^+ \in \Int^+$. We discuss this link in the section that follows.

\subsection{Wajsberg's Theorem about $\zero$-elemination}

If $\pi$ is a finite non-void set of variables, by $\pi^\land$ we denote the conjunction of all variables from $\pi$, that is, $\pi^\land \bydef \bigwedge_{p \ \in \ \pi} p$.

In \cite{Wajsberg_Heyting_1977} Wajsberg was using the following reduction: if $A$ is a formula and $p_0,\dots,p_{n-1}$ is a set of all variables occurring in $A$ and $p$ is a variable not occurring in $A$, then formula $A$ can be reduced to the positive formula 
\[
A^* \bydef (p \to p_1) \to ((p \to p_2) \to ((p \to p_3) \dots \to A')\dots),   
\]
where $A'$ is obtained from $A$ by replacing all occurrences of $\zero$ with $p$. 

Let $A$ be a formula and $\pi$ be a finite non-void set of variables.
A \textit{reduction of} $A$ \textit{by} $\pi$ is the formula $A^\pi$ obtained from $A$ by replacing every occurrence of $\zero$ with $\pi^\land$. If $\pi =\{p\}$ instead of $A^{\{p\}}$ we write $A^p$.

We use the following modification of Wajsberg reduction (comp. \cite{Verhozina_Intermediate_1978}). 

\begin{definition} Let $\pi$ be a set of variables and $p$ be a variable.
\textit{Wajsberg reduction} of a formula $A$ \emph{by} $p,\pi$ is the formula
\begin{equation}
W(A,\pi,p) \means (p \to \pi^\land) \to A^p.  \label{wreduct}
\end{equation}
\end{definition}

Now, the Wajsberg's Theorem can be stated as follows. 

\begin{theorem}[{\cite{Wajsberg_Heyting_1977}[Theorem 1 \S7]}] \label{WajsTh} For any formula $A$, 
\[
\vdash_\IPC A \text{ if and only if } \vdash_{\IPC^+} W(A,\pi,p),
\]
where $p \notin \pi(A)$,$\pi \bydef \pi(A) \cup \{p\},$ and $\IPC^+$ is a calculus obtained from $\IPC$ by omitting axioms for $\zero$.  
\end{theorem}

\subsection{Generalization of Wajsberg Theorem} \label{WajsbergTh}

The below theorem extends Theorem \ref{WajsTh} to the logics corresponding to B-saturated varieties. As we will see (Theorem \ref{thposdef}), the logics corresponding to B-saturated varieties are precisely the logics that can be defined by a set of positive formulas. 

First, we need to establish a rather simple technical property of the Wajsberg's reduction.

\begin{prop} \label{prsubs}
Let $A$ be a positive formula and $\pi \supseteq \pi(A)$ be a finite non-void set of variables. Then for every substitution $\sigma \in \Sigma$ there is a substitution $\sigma^\pi \in \Sigma^+$ such that
\[
(\sigma(A))^\pi = \sigma^\pi(A).
\]
\end{prop}
\begin{proof} Suppose $\pi(A) = \{p_0,\dots,p_{n-1}\}$, i.e. $A = A(p_0,\dots,p_{n-1})$ and $\sigma: p_i \mapsto B_i, i < n$. Then, $\sigma(A) = A(B_1,\dots,B_n)$, and
\[
(\sigma(A))^\pi = (A(B_0,\dots,B_{n - 1}))^\pi. 
\]
Recall that $A$ is a positive formula and, thus, does not contain $\zero$. Therefore, 
\[
(A(B_0,\dots,B_{n - 1}))^\pi  = A(B_0^\pi,\dots,B_{n -1}^\pi). 
\]
Let us observe that $B^\pi_i \in \Frm^+$ for all $i < n$. Therefore,
\[
(\sigma(A))^\pi = A(B_1^\pi,\dots,B_n^\pi) = \sigma^\pi(A), 
\] 
where
\[
\sigma^\pi: p_i \mapsto B_i^\pi, i < n. 
\]
And it is clear that $\sigma^\pi \in \Sigma^+$.
\end{proof}

\begin{theorem} \label{thwajsb}Let $\var$ be a B-saturated variety, $A$ be a formula and $\pi$ be a set of variables such that $\pi(A) \subset \pi$, and let $p \in \pi \setminus \pi(A)$. Then the following is equivalent
\begin{itemize}
\item[(a)] $\var \models A$;
\item[(b)] $\var \models W(A,\pi,p)$;
\item[(c)] $\var \models A^\pi$.
\end{itemize}
\end{theorem}

\begin{remark} The equivalence of (a) and (c) was observed by M.~Verhozina \cite{Verhozina_Intermediate_1978}[Corollary at p.16]. See also  \cite{deJongh_Zhao_Positive_2015} where a different reduction to a positive formula is used.
\end{remark}

\begin{proof}
(b) $\Rightarrow$ (c).  $A^\pi$ can be derived from $W(A,\pi,p)$ by substitution of $\pi^\land$ for $p$ and applying modus ponens (see \eqref{wreduct}). 

(c) $\Rightarrow$ (a). $A$ can be derived from $A^\pi$ by substitution of $\zero$ for $p$. Since $p$ does not occur in $A$, the obtained formula is equivalent to $A$. 

(a) $\Rightarrow$ (b). Suppose $\var \not\models W(A,\pi,p)$ and let us demonstrate that $\var \not\models A$. 

Indeed, if  $\var \not\models W(A,\pi,p)$, there is a Heyting algebra $\Alg{A} \in \var$ and a valuation $\nu$ in $\Alg{A}$ such that
\begin{equation}
\nu(p \to \pi^\land) = \one \text{ and } \nu(A^p) \neq \one. \label{eq_thwajsb1}
\end{equation}

Let us consider a Brouwerian subreduct $\Alg{B}$ of algebra $\Alg{A}$ generated by elements $\nu(q), q \in \pi$ and $\nu(p)$. From \eqref{eq_thwajsb1} and Proposition \ref{leastel} it follows that $\nu$ maps $p$ to the least element of $\Alg{B}$. Let us consider $\Alg{B}^\zero$ and note that $\Alg{B}$ and $\Alg{B}^\zero$ have the same universes. It is not hard to see that $\Alg{B}^\zero \models A$. Observe that $\Alg{B}^\zero \in^+ \var$ and recall that $\var$ is a B-saturated variety. Hence, $\Alg{B}^\zero \in \var$. Since $\Alg{B}$ and $\Alg{B}^\zero$ have the same universes, we can view $\nu$ as a valuation in $\Alg{B}^\zero$ and it is not hard to see that $\nu(p) = \zero$. Thus, $\nu$ is a valuation refuting $A$ in $\Alg{B}^\zero$, that is, $\var \not\models A$.
\end{proof}

\subsection{Positively Defined Superintuitionistic Logics.} \label{PosDefined} In this sections we consider si-logics corresponding to B-saturated varieties.

\begin{definition}
We say that an si-logic $\LogL$ is \textit{positively defined} (\textit{p-defined} for short) if $\LogL$ can be defined by $\IPC$ extended by a set (not necessarily finite) of positive formulas. 
\end{definition}

Our goal is to prove the following theorem.

\begin{theorem} \label{thposdef} An superintuitionistic logic is positively defined if and only if the corresponding variety is B-saturated.  
\end{theorem}

It is clear that a logic $\LogL$ is p-defined if and only if its corresponding variety $\var_\LogL$ can be defined over $\Heyt$ by a set of positive formlas. Thus, Theorem \ref{thposdef} is equivalent to the following theorem.

\begin{theorem} \label{thstur} A variety $\var \subseteq \Heyt$ is B-saturated if and only if $\var$ can be defined over $\Heyt$ by a set of positive formulas.
\end{theorem}
\begin{proof} Let $\var$ be a B-saturated variety. Let us prove that $\var$ can be defined over $\Heyt$ by the set $\mathcal{F}$ of all positive formulas valid in $\var$. 

Indeed, let $\var'$ be a variety defined over $\Heyt$ by $\mathcal{F}$. Clearly, $\var'$ is the greatest variety in which all formulas from $\mathcal{F}$ are valid. So, since every formula from $\mathcal{F}$ is valid in $\var$, we have $\var \subseteq \var'$, and  all we need to prove is $\var' \setminus \var = \varnothing$. 

Indeed, assume to the contrary that $\var' \setminus \var \neq \varnothing$. Then there is a formula $A$, such that 
\begin{equation}
\var \models A \text{, while } \var' \not\models A. \label{eq_ththstur1}
\end{equation}
Suppose that $\pi$ is a finite set of variables such that $\pi(A) \subset \pi$. Due to $\var$ is B-saturated variety, from $\var \models A $, by Theorem \ref{thwajsb}, we conclude $\var \models A^\pi$. Let us observe that $A^\pi$ is a positive formula and, hence, $A^\pi \in \mathcal{F}$. But $\var'$ is defined by formulas $\mathcal{F}$, hence $\var' \models A^\pi$. Note that $\var' \models A^\pi$ entails 
\begin{equation}
\var' \models A,  \label{eq_ththstur2}
\end{equation} 
because $A$ can be derived in $\Int$ from $A^\pi$ by substituting $p$ with $\zero$. And \eqref{eq_ththstur2} contradicts \eqref{eq_ththstur1}.
\end{proof}

Let us note that Theorem \ref{thwajsb} (more precisely the equivalence of (a) and (b)) is a generalization of \cite{Wajsberg_Heyting_1977}[Theorem 1 of \$7]  to p-defined logics. \\

\begin{cor} \label{corunif} Let $\LogP$ be a positive logic, $\LogL \bydef \Int + \LogP$ and $\Gamma$ be a set of positive formulas. Then $\Gamma$ is $\LogP$-unifiable
 if and only if $\Gamma$ is $\LogL$-unifiable.
\end{cor}
\begin{proof} If $\Gamma$ is $\LogP$-unifiable, then any $\LogP$-unifier of $\Gamma$ is at the same time a $\LogL$-unifier. Hence, $\Gamma$ is $\LogL$-unifiable.

Conversely, suppose that $\sigma$ is a $\LogL$-unifier of $\Gamma$ and $\Gamma \bydef \{A_i, i <n\}$. Thus $\vdash_\LogL \sigma(A_i)$ for all $i<n$. Let $\pi$ be set of variables strongly containing a set of all variables occurring in all formulas $A_i, i <n$. Logic $\LogL$ is positively defined, hence, by Theorem \ref{thstur}, variety $\var_\LogL$ is B-saturated and we can apply Theorem \ref{thwajsb} and conclude that $\vdash_\LogL \sigma(A_i)^\pi$ for all $i <n$. Recall that all formulas $A_i$ are positive and, by Proposition \ref{prsubs}, there is a positive substitution $\sigma^\pi$ such that $\sigma(A_i)^\pi = \sigma^\pi(A), i < n$. 

by Corollary \ref{corposf}, 
\end{proof}

In \cite{Wronski_1973_Pos} Wro{\'n}ski observed that the cardinality of $\Ext{\Int^+}$ is that of continuum. From Corollary \ref{cordist} it follows that the cardinality of the set of p-defined logics is not less than cardinality of $\Ext{\Int}^+$. Thus, the following holds.

\begin{cor}\label{corcont} There is continuum many p-defined logics and, hence, there is continuum many p-saturated varieties of Heyting algebras.
\end{cor}

\begin{example}  Recall from \cite{Jankov_Calculus_1968} that all si-logics between $\Int$ and $\Int + \neg p \lor \neg\neg p$ have the same positive fragment . Hence, besides $\Int$, neither of these logics is p-defined, and, consequently,
besides $\Heyt$ neither extension of $\KC$ is B-saturated.
\end{example}

\section{Admissibility of Rules in Positive Logics}

The goal of this section is to prove the following theorem.

\begin{theorem} \label{thmadm} Let $\LogP$ be a positive logic, $\ruleR \bydef \Gamma/\Delta$ be a positive m-rule and $\LogL \bydef \Int + \LogP$. Then 
\[
\Padm \ruleR  \text{ if and only if } \Ladm \ruleR.
\]
\end{theorem}
\begin{proof} Suppose that $\Ladm \ruleR$. Consider four cases:
\begin{itemize}  
\item[(a)] $\Gamma \neq \varnothing$ and $\Delta \neq \varnothing$; 
\item[(b)] $\Gamma \neq \varnothing$ and $\Delta = \varnothing$;
\item[(c)] $\Gamma = \varnothing$ and $\Delta \neq \varnothing$;
\item[(d)] $\Gamma = \varnothing$ and $\Delta = \varnothing$.
\end{itemize}

(a) Suppose that $\Gamma \neq \varnothing$ and $\Delta \neq \varnothing$. Then any substitution $\sigma \in \Sigma$ that $\LogL$-unifies $\Gamma$ at the same time $\LogL$-unifies at least one formula from $\Delta$. Thus, if $\sigma^+$ is a positive substitution that $\LogP$-unifies $\Gamma$, we have $\sigma^+(B) \in \LogL$ for some formula $B \in \Delta$. Recall that $B$ is a positive formula. Therefore, $\sigma^+(B)$ is a positive formula and, by Corollary \ref{corposf}, $\sigma^+(B) \in \LogP$ if and only if $\sigma^+(B) \in (\Int + \LogP)^+ \subseteq \LogL$.
Hence, if $\sigma^+$ $\LogP$-unifies $B$, that is, m-rule $\ruleR$ is admissible for $\LogP$. 

(b) If $\Gamma \neq \varnothing$ and $\Delta = \varnothing$, then $\LogL$-admissibility of $\ruleR$ means that neither substitution $\LogL$-unifies $\Gamma$, which, in turn means that neither positive substitution $\LogP$-unifies $\Gamma$.   

(c) If $\Gamma = \varnothing$ and $\Delta \neq \varnothing$, then $\LogL$-admissibility of $\ruleR$ means that at least one formula $B$ from $\Delta$ is in $\LogL$. Due to $B$ is positive, $B \in \LogL^+$ and, by Corollary \ref{corposf}, $ \LogL^+ = \LogP$, i.e. $B \in \LogP$. Hence, $\ruleR$ is $\LogP$-admissible.

(d) If $\Gamma = \varnothing$ and $\Delta = \varnothing$, then $\LogL$-admissibility means that $\LogL$ is inconsistent, that is, $\LogL = \Frm$, therefore, $\LogP = \LogL^+ = \Frm^+$, and $\ruleR$ is $\LogP$-admissible.\\

Conversely, suppose that $\NLadm \ruleR$. We need to show that $\NPadm \ruleR$. We again consider four above cases.

(a) If $\NLadm \ruleR$, there is a substitution $\sigma$ that $\LogL$-unifies $\Gamma \bydef \{A_i, i < n \}$ but not $\LogL$-unifies any formula from $\Delta \bydef \{ B_j, j < m \}$. Let $\pi$ be a set of variables such that $\bigcup_{i<n}\pi(A_i) \cup \bigcup_{j<m}\pi(B_j) \subset \pi$, i.e. $\pi$ contains all variables from all formulas from $\ruleR$ and at least one extra variable. Logic $\Int + \LogP$ is positively defined, hence, by Theorem \ref{thposdef}, the corresponding variety $\var$ is B-saturated. Recall that by Theorem \ref{thwajsb}, for any positive formula $C$, $\var \models C$ if and only if $\var \models C^\pi$. Hence, for all $i < n$
\[
\var \models \sigma(A_i) \text{ if and only if } \var \models \sigma(A_i)^\pi
\]
and, by Proposition \ref{prsubs}, there is a positive substitution $\sigma^\pi$ such that
\[
\var \models \sigma(A_i)^\pi \text{ if and only if } \var \models \sigma^\pi(A_i). 
\]
Thus $\sigma^\pi$ is a $\LogP$-unifier for $\Gamma$. A similar argument shows that $\sigma^\pi$ does not $\LogP$-unify any formula $B_j, j <m$, which means that $\ruleR$ is not admissible for $\LogP$, that is, $\NPadm \ruleR$.

(b) If $\Delta = \varnothing$ and $\NLadm \ruleR$, it simply means that $\Gamma$ is $\LogL$-unifiable. We can repeat the argument from case (a) and conclude that $\Gamma$ is $P$-unifiable, hence, rule $\Gamma/\varnothing$ is not admissible for $\LogP$.  

(c) If $\NLadm \varnothing/\Delta$, then $\Delta \cap \LogL = \varnothing$ and, due to all formulas from $\Delta$ are positive, $\Delta \cap \LogL^+ = \varnothing$. Recall that $\LogL^+ = \LogP$, hence, $\Delta \cap \LogP = \varnothing$, which means $\NPadm \varnothing/\Delta$.

(d) $\NLadm \varnothing/\varnothing$ means that $\LogL$ is consistent. Hence, $p \notin \LogL$, where $p$ is a variable. But $p$ is a positive formula, i.e. $\LogL^+ \neq \Frm^+$. Thus, we have $\LogP = \LogL^+ \neq \Frm^+$, which means that $\LogP$ is consistent and $\NPadm \varnothing/\varnothing$. 
\end{proof}

\section{Derivability of Rules in Positive Logics}

\begin{definition} Let $\Rules$ be a set of m-rules, $\ruleR$ be an m-rule and $\var$ be a variety. Then $\ruleR$ (semantically) $\var$-\textit{follows from }$\Rules$ (in symbols $\Rules \models_\var \ruleR$) if
\[
\Alg{A} \models \Rules \text{ entails } \Alg{A} \models \ruleR  \text{ for every } \Alg{A} \in \var. 
\]
\end{definition}

In a natural way, the above definition can be extended to the logics.

\begin{definition} Let $\Rules$ be a set of m-rules, $\ruleR$ be an m-rule and $\LogL$ be a logic. Then $\ruleR$ (semantically) $\LogL$-\textit{follows from }$\Rules$ (in symbols $\Rules \models_\LogL \ruleR$) if
$\Rules \models_{\var_\LogL} \ruleR$, that is, $\ruleR$ follows from $\Rules$ where $\var$ w.r.t. variety corresponding to $\LogL$.
\end{definition}

\begin{theorem} \label{thfollow} Let $\LogP$ be a positive logic and $\LogL \bydef \Int + \LogP$. If $\Rules$ is a set of positive m-rules and $\ruleR$ is a positive m-rule, then
\[
\Rules \models_\LogP \ruleR \text{ if and only if } \Rules \models_{\LogL} \ruleR.
\]
\end{theorem}
\begin{proof} Suppose that $\Rules \not\models_\LogL \ruleR$. Then there is a Heyting algebra $\Alg{A}  \in \var_\LogL$ such that $\Alg{A} \models \Rules$ and $\Alg{A} \not\models \ruleR$. Due to $\Alg{A} \in \var_\LogL$, all formulas from $\LogP$ are valid in $\Alg{A}$ and, hence, $\Alg{A}^+ \in \var_\LogP$. Since all rules that we consider are positive, $\Alg{A}^+ \models \Rules$ and $\Alg{A}^+ \not\models \ruleR$. Thus, $\Rules \not\models_\LogP \ruleR$.  

Conversely, suppose that $\Rules \not\models_\LogP \ruleR$. Then there is a Brouwerian algebra $\Alg{B} \in \var_\LogP$ such that $\Alg{B} \models \Rules$ and $\Alg{B} \not\models \ruleR$. Since $\ruleR$ contains only finite number of variables, without loosing generality we can assume that $\Alg{B}$ is finitely generated. Hence, $\Alg{B}$ has the least element and $\Alg{B}$ can be viewed as a Heyting algebra that we denote $\Alg{B}^\zero$. Due to $\Alg{B} \in \var_\LogP$, all formulas from $\LogP$ are valid in $\Alg{B}$, and, hence, all formulas from $\LogP$ are valid in $\Alg{B}^\zero$. Thus, $\Alg{B}^\zero \in \var_\LogL$. Now, recall that $\Rules$ and $\ruleR$ are positive rules, therefore, $\Alg{B}^\zero \models \Rules$ and $\Alg{B}^\zero \not\models \ruleR$, and this means that $\Rules \not\models_\LogL \ruleR$. 
\end{proof}

The above theorem has two useful corollaries. First, we recall the notion of independence of a set of m-rules.

\begin{definition} Let $\LogL$ be a logic. A set of m-rules $\Rules$ is called $\LogL$\textit{-independent} if no rule $\ruleR \in \Rules$ $\LogL$-follows from $\Rules \setminus \{\ruleR\}$. 
\end{definition}

\begin{cor} \label{corindep} Let $\LogP$ be a positive logic and $\LogL \bydef \Int + \LogP$. Then a set $\Rules$ of positive m-rules is $\LogP$-independent if and only if $\Rules$ is $\LogL$-independent.
\end{cor}

\begin{definition} Let $\LogL$ be a logic. A set of admissible m-rules $\Rules$ is called a \textit{basis of $\LogL$-admissible rules} if any admissible for $\LogL$ rule $\LogL$-follows from $\Rules$. A basis $\Rules$ is called independent if $\Rules$ is $\LogL$-independent set.
\end{definition}

\begin{cor} \label{corbasis} Let $\LogP$ be a positive logic and $\LogL \bydef \Int + \LogP$. If a set of positive m-rules $\Rules$ forms an $\LogL$-basis of $\LogL$-admissible rules, then $\Rules$ forms a $\LogP$-basis of $\LogP$-admissible rules. If $\Rules$ is a $\LogL$ independent basis, the $\Rules$ is $\LogP$-independent as well. 

If $\Rules$ is a basis of $\LogP$-admissible m-rules, then any admissible for $\LogL$ positive m-rule $\LogL$-follows from $\Rules$.
\end{cor}

\begin{example} An independent basis for $\Int$ consisting of positive m-rules had been constructed in \cite{Jerabek_Independent_2008}[Theorem 3.12]. This set of m-rules forms an independent basis for admissible $\Int^+$-rules.
\end{example}


\bibliographystyle{acm}

\begin{thebibliography}{10}

\bibitem{Belnap_et_Strengthening_1963}
{\sc Belnap, Jr., N.~D., Leblanc, H., and Thomason, R.~H.}
\newblock On not strengthening intuitionistic logic.
\newblock {\em Notre Dame J. Formal Logic 4\/} (1963), 313--320.

\bibitem{Birkhoff}
{\sc Birkhoff, G.}
\newblock {\em Lattice {T}heory}.
\newblock American Mathematical Society Colloquium Publications, vol. 25,
  revised edition. American Mathematical Society, New York, N. Y., 1948.

\bibitem{Cintula_Metcalfe_Admissible_2010}
{\sc Cintula, P., and Metcalfe, G.}
\newblock Admissible rules in the implication-negation fragment of
  intuitionistic logic.
\newblock {\em Ann. Pure Appl. Logic 162}, 2 (2010), 162--171.

\bibitem{Citkin1977}
{\sc Citkin, A.}
\newblock On admissible rules of intuitionistic propositional logic.
\newblock {\em Math. USSR, Sb. 31\/} (1977), 279--288.
\newblock (A.~Tsitkin).

\bibitem{Citkin_Positive_1988}
{\sc Citkin, A.}
\newblock On admissibility of rules in positive logic.
\newblock In {\em IX All-Union Conference for Mathematical Logic}. ``Nauka'',
  1988, p.~171.
\newblock (in {R}ussian).

\bibitem{deJongh_Zhao_Positive_2015}
{\sc de~Jongh, D., and Zhao, Z.}
\newblock Positive formulas in intuitionistic and minimal logic.
\newblock In {\em Logic, Language and Computation}, vol.~8984 of {\em Lecture
  Notes in Computer Science}. Springer-Verlag, 2015, pp.~175--189.
\newblock 10th International Tbilisi Symposium on Logic, Language, and
  Computation, TbiLLC 2013, Gudauri, Georgia, September 23-27, 2013. Revised
  Selected Papers.

\bibitem{GorbunovBookE}
{\sc Gorbunov, V.~A.}
\newblock {\em Algebraic theory of quasivarieties}.
\newblock Siberian School of Algebra and Logic. Consultants Bureau, New York,
  1998.
\newblock Translated from the Russian.

\bibitem{Goudsmit_PhD}
{\sc Goudsmit, J.}
\newblock {\em Intuitionistic Rules Admissible Rules of Intermediate Logics}.
\newblock PhD thesis, Utrech University, 2015.

\bibitem{Goudsmit_Iemhoff_Unification_2014}
{\sc Goudsmit, J.~P., and Iemhoff, R.}
\newblock On unification and admissible rules in {G}abbay--de {J}ongh logics.
\newblock {\em Ann. Pure Appl. Logic 165}, 2 (2014), 652--672.

\bibitem{Harrop_Concerning_1960}
{\sc Harrop, R.}
\newblock Concerning formulas of the types {$A\rightarrow B\bigvee
  C,\,A\rightarrow (Ex)B(x)$} in intuitionistic formal systems.
\newblock {\em J. Symb. Logic 25\/} (1960), 27--32.

\bibitem{Horn_Separation_1962}
{\sc Horn, A.}
\newblock The separation theorem of intuitionist propositional calculus.
\newblock {\em J. Symbolic Logic 27\/} (1962), 391--399.

\bibitem{Jankov_Calculus_1968}
{\sc Jankov, V.~A.}
\newblock Calculus of the weak law of the excluded middle.
\newblock {\em Izv. Akad. Nauk SSSR Ser. Mat. 32\/} (1968), 1044--1051.
\newblock English translation in Math. of the USSR-Izvestiya, 2:5, 997â??1004
  (1968).

\bibitem{Jerabek_Independent_2008}
{\sc Je{\v{r}}{\'a}bek, E.}
\newblock Independent bases of admissible rules.
\newblock {\em Log. J. IGPL 16}, 3 (2008), 249--267.

\bibitem{Jerabek_Canonical_2009}
{\sc Je{\v{r}}{\'a}bek, E.}
\newblock Canonical rules.
\newblock {\em J. Symbolic Logic 74}, 4 (2009), 1171--1205.

\bibitem{Kohler_Varieties_1975}
{\sc K{\"o}hler, P.}
\newblock Varieties of {B}rouwerian algebras.
\newblock {\em Mitt. Math. Sem. Giessen}, 116 (1975), iii+83.

\bibitem{Kracht_Review_1999}
{\sc Kracht, M.}
\newblock Book review of \cite{Rybakov_Book}.
\newblock {\em Notre Dame J. Form. Log. 40}, 4 (1999), 578 -- 587.

\bibitem{Kracht_Modal_2007}
{\sc Kracht, M.}
\newblock Modal consequence relations.
\newblock In {\em Handbook of Modal Logic}, P.~Blackburn and et~al., Eds.,
  vol.~3 of {\em Studies in Logic and Practical Reasonong}. Elsevier, 2007,
  ch.~8, pp.~491 -- 545.

\bibitem{Lorenzen_Book_1955}
{\sc Lorenzen, P.}
\newblock {\em Einf\"uhrung in die operative {L}ogik und {M}athematik}.
\newblock Die Grundlehren der mathematischen Wissenschaften in
  Einzeldarstellungen mit besonderer Ber\"ucksichtigung der Anwendungsgebiete,
  Bd. LXXVIII. Springer-Verlag, Berlin-G\"ottingen-Heidelberg, 1955.

\bibitem{Lorenzen_Protologik_1956}
{\sc Lorenzen, P.}
\newblock Protologik. {Ein Beitrag zum BegrÃ¼ndungsproblem der Logik}.
\newblock {\em Kant-Studien 47}, 1 -- 4 (Jan 1956), 350 -- 358.
\newblock Translated in P. Lorenzen Constructive Philosophy (Univerisity of
  Massachusettes Press, Amherst, 1987, pp. 59 -- 70).

\bibitem{Mints1971}
{\sc Mints, G.}
\newblock {Derivability of admissible rules.}
\newblock {\em J. Sov. Math. 6\/} (1976), 417--421.
\newblock Translated from Mints, G. E. Derivability of admissible rules.
  (Russian) Investigations in constructive mathematics and mathematical logic,
  V. Zap. Nauchn. Sem. Leningrad. Otdel. Mat. Inst. Steklov. (LOMI) 32 (1972),
  pp. 85 - 89.

\bibitem{Monteiro_Axioms_1955}
{\sc Monteiro, A.}
\newblock Axiomes ind\'ependants pour les alg\`ebres de {B}rouwer.
\newblock {\em Rev. Un. Mat. Argentina 17\/} (1955), 149--160 (1956).

\bibitem{Novikov_Book}
{\sc Novikov, P.~S.}
\newblock {\em Konstruktivnaya matematicheskaya logika s tochki zreniya
  klassicheskoi [Constructive mathematical logic from the point of view of
  classical logic]}.
\newblock Izdat. ``Nauka'', Moscow, 1977.
\newblock With a preface by S. I. Adjan, Matematicheskaya Logika i Osnovaniya
  Matematiki. [Monographs in Mathematical Logic and Foundations of Mathematics]
  (in Russian).

\bibitem{Odintsov_Rybakov_Unification_2013}
{\sc Odintsov, S., and Rybakov, V.}
\newblock Unification and admissible rules for paraconsistent minimal
  {J}ohanssons' logic {$\mathbf{J}$} and positive intuitionistic logic
  {$\mathbf{IPC}^+$}.
\newblock {\em Ann. Pure Appl. Logic 164}, 7-8 (2013), 771--784.

\bibitem{Odintsov_Constructive_2008}
{\sc Odintsov, S.~P.}
\newblock {\em Constructive negations and paraconsistency}, vol.~26 of {\em
  Trends in Logic---Studia Logica Library}.
\newblock Springer, New York, 2008.

\bibitem{Prucnal_On_Structural_1972}
{\sc Prucnal, T.}
\newblock On the structural completeness of some pure implicational
  propositional calculi.
\newblock {\em Studia Logica 30\/} (1972), 45--52.

\bibitem{Rasiowa_Algebraic_1974}
{\sc Rasiowa, H.}
\newblock {\em An algebraic approach to non-classical logics}.
\newblock North-Holland Publishing Co., Amsterdam, 1974.
\newblock Studies in Logic and the Foundations of Mathematics, Vol. 78.

\bibitem{Rybakov_Criterion_Adm_1984}
{\sc Rybakov, V.~V.}
\newblock A criterion for admissibility of rules in the modal system {${\rm
  S}4$} and intuitionistic logic.
\newblock {\em Algebra i Logika 23}, 5 (1984), 546--572, 600.

\bibitem{Rybakov_Bases_Adm_1985}
{\sc Rybakov, V.~V.}
\newblock Bases of admissible rules of the logics {${\rm S}4$} and {${\rm
  Int}$}.
\newblock {\em Algebra i Logika 24}, 1 (1985), 87--107, 123.

\bibitem{Rybakov_Book}
{\sc Rybakov, V.~V.}
\newblock {\em Admissibility of logical inference rules}, vol.~136 of {\em
  Studies in Logic and the Foundations of Mathematics}.
\newblock North-Holland Publishing Co., Amsterdam, 1997.

\bibitem{Verhozina_Intermediate_1978}
{\sc Verhozina, M.}
\newblock Intermediate positive logics.
\newblock In {\em Algorithmic Problems of Algebraic Systems}. Irkutsk State
  University, 1978, pp.~13 -- 25.
\newblock (in {R}ussian).

\bibitem{Wajsberg_Heyting_1977}
{\sc Wajsberg, M.}
\newblock On {A. H}eyting's propositional calculus.
\newblock {\em Waysberg M. Logical works. Warszava\/} (1977), 132--171.
\newblock Translated from: "Untersuchungen uber den Aussagenkalkul von A.
  Heyting" , Wiadomo{\'s}ci matematyczne 46 (1938), pp. 45-101.

\bibitem{Wronski_1973_Pos}
{\sc Wro{\'n}ski, A.}
\newblock On the degree of completeness of positive logic.
\newblock {\em Polish Acad. Sci. Inst. Philos. Sociol. Bull. Sect. Logic 2}, 1
  (1973), 65--70.

\end{thebibliography}

\def\cprime{$'$}

\end{document}